\documentclass{amsart}
\usepackage{amsthm}
\usepackage{amssymb}
\usepackage{mathtools}

\newcommand{\GAP}{\textsf{GAP}}
\newcommand{\MAGMA}{\textsf{Magma}}

\newcommand{\N}{\mathbb{N}}
\newcommand{\Z}{\mathbb{Z}}

\newcommand{\B}{\mathbb{B}}

\newcommand{\Sym}{\mathbb{S}}

\newcommand{\id}{\mathrm{id}}

\newcommand{\Aut}{\mathrm{Aut}}
\newcommand{\Hol}{\mathrm{Hol}}
\newcommand{\Soc}{\mathrm{Soc}}

\makeatletter
\numberwithin{table}{section}
\numberwithin{equation}{section}
\numberwithin{figure}{section}
\theoremstyle{plain}
\newtheorem{thm}{Theorem}[section]
\theoremstyle{plain}
\newtheorem{lem}[thm]{Lemma}
\newtheorem{defn}[thm]{Definition}
\newtheorem{cor}[thm]{Corollary}
\theoremstyle{plain}
\newtheorem{pro}[thm]{Proposition}
\newtheorem{conjecture}[thm]{Conjecture}
\newtheorem{question}[thm]{Question}
\newtheorem{algorithm}[thm]{Algorithm}
\newtheorem{problem}[thm]{Problem}
\newtheorem{example}[thm]{Example}
\theoremstyle{remark}
\newtheorem{rem}[thm]{Remark}

\theoremstyle{plain}

\makeatother

\begin{document}

\title[Skew braces and the Yang--Baxter equation]{Skew braces and the Yang--Baxter equation}
\thanks{This work is partially supported by CONICET, PICT-2014-1376, MATH-AmSud and ICTP}

\begin{abstract}
    Braces were introduced by Rump to study non-degenerate involutive
    set-theoretic solutions of the Yang--Baxter equation.  We generalize Rump's
    braces to the non-commutative setting and use this new structure to study
    not necessarily involutive non-degenerate set-theoretical solutions of the
    Yang--Baxter equation.  Based on results of Bachiller and Catino and Rizzo,
    we develop an algorithm to enumerate and construct classical and
	non-classical braces of small size up to isomorphism. This algorithm is
    used to produce a database of braces of small
    size.  The paper contains several open problems, questions and conjectures. 
\end{abstract}

\author{L. Guarnieri}
\author{L. Vendramin}
\address{
Departamento de Matem\'atica -- FCEN,
Universidad de Buenos Aires, Pab. I -- Ciudad Universitaria (1428)
Buenos Aires -- Argentina}
\email{leandroguarnieri@gmail.com}
\email{lvendramin@dm.uba.ar}

\maketitle

\section*{Introduction}

The Yang--Baxter equation first appeared in theoretical physics and statistical mechanics in the works
of Yang~\cite{MR0261870} and Baxter~\cite{MR0290733,MR998375} and it has led to
several interesting applications in quantum groups and Hopf algebras, knot
theory, tensor categories and integrable systems, see for
example~\cite{MR1321145},~\cite{MR1381692} and~\cite{MR2024436}.  In
\cite{MR1183474}, Drinfeld posed the problem of studying this equation from the
set-theoretical perspective. 

Recall that a set-theoretical solution of the Yang--Baxter equation is a pair
$(X,r)$, where $X$ is a set and 
\[
r\colon
X\times X\to X\times X,\quad
r(x,y)=(\sigma_x(y),\tau_y(x)),
\quad
x,y\in X,
\]
is a bijective map such that
\[
    (r\times\id)(\id\times r)(r\times\id)=(\id\times r)(r\times\id)(\id\times r).
\]
Such a map $r$ is usually called a braiding. 

A solution $(X,r)$ is said to be non-degenerate if the maps $\sigma_x$ and
$\tau_x$ are bijective for each $x\in X$, and $(X,r)$ is said to be involutive
if $r^2=\id_{X\times X}$.  The seminal works of Etingof, Schedler and
Soloviev~\cite{MR1722951}, and Gateva-Ivanova and Van den
Bergh~\cite{MR1637256}, discussed algebraic and geometrical interpretations and
introduced several structures associated with the class of non-degenerate
involutive solutions.   Such solutions have been intensively studied, see for example
~\cite{GI,MR2095675,MR2927367},~\cite{MR2368074,MR2383056,MR2776789},~\cite{MR2189580,MR2301033},
~\cite{MR2885602}, ~\cite{MR2132760,MR2278047},
~\cite{MR2534251},
~\cite{MR2584610},
~\cite{MR2652212},
~\cite{MR3374524},
~\cite{LV},
and~\cite{MR3437282}.

It was in studying involutive solutions that Rump introduced in
\cite{MR2278047} the brace structure. In~\cite{MR3177933}, Ced\'o, Jespers and
Okni{\'n}ski, defined  a left brace as an abelian group $(A,+)$ with another
group structure, defined via $(a,b)\mapsto ab$, such that the
compatibility condition
\[
a(b+c)+a=ab+ac
\]
holds for all $a,b,c\in A$. This definition is equivalent to that of Rump.

Many of the problems related to involutive solutions can be restated in terms
of braces.  Two prominent examples are the following: 
\begin{itemize}
    \item Is every finite solvable group an involutive Yang--Baxter group?
        Recall that an involutive Yang--Baxter group is a group isomorphic to
        the group generated by the set $\{\sigma_x:x\in X\}$, where 
        \[
        r\colon
        X\times X\to X\times X,
        \quad
        r(x,y)=(\sigma_x(y),\tau_y(x)),
        \]
        is a non-degenerate involutive solution of the Yang--Baxter equation.
        Based on a sketch of proof of Rump~\cite{MR3291816},
        Bachiller~\cite{MR3465351} found a solvable finite group that is not an
        involutive Yang--Baxter group.
    \item Are there good methods to contruct all finite
		non-degenerate involutive solutions to the Yang--Baxter equation?
		Brute force seems not to be good enough.  In \cite{BCJ}, Bachiller,
		Ced\'o and Jespers, give a method to construct all finite solutions of
		a given size. For it 
		to work, one needs the classification of
		left braces.
\end{itemize}

Non-involutive solutions were studied by Soloviev~\cite{MR1809284} and Lu, Yan
and Zhu~\cite{MR1769723}.  Such solutions have applications in knot theory,
since they produce powerful knot and virtual knots invariants, see for
example~\cite{MR2896084} and the references therein.  The following question
naturally arises: Is there an algebraic structure similar to the brace
structure useful for studying non-involutive solutions? This paper introduces
the notion of \emph{skew brace} and provides an affirmative answer
to the above question.  
Remarkably, this new structure provides the right algebraic framework to study
involutive and non-involutive braidings and allows us to restate the main
results of~\cite{MR1769723},~\cite{MR1809284} and~\cite{MR1178147}. 

As in the case of involutive solutions, the classification
of finite skew braces is one of the main steps needed for constructing 
finite solutions of the Yang--Baxter equation. 
One of the main results of this paper is an explicit classification of 
classical and skew braces of small size.  An algorithm to construct all
non-isomorphic classical and skew braces of a given size is described. 
This heavily depends  on results of Bachiller~\cite{MR3465351} and Catino and
Rizzo~\cite{MR2486886}. This algorithm was used to build a database of
classical and skew braces, a good source of examples that gives an
explicit and direct way to approach some of the problems related to the
Yang--Baxter equation. The database is available as a library
for~\GAP~\cite{GAP4}~and~\MAGMA~\cite{MR1484478} immediately from the authors
on request. 

\medskip
The paper is organized as follows.  In Section~\ref{braces} we extend braces to
the non-commutative setting by defining skew braces,
and state their main properties. We prove in Proposition~\ref{pro:1cocycle} that skew braces are equivalent to
bijective $1$-cocycles. 
Section~\ref{ideals} is devoted to a study of 
quotients of skew braces. 
It is worth mentioning that the proofs in Section~\ref{braces}~and~\ref{ideals}
are basically the same as for classical braces.  
In Section~\ref{YB} the connection between
skew braces and the Yang--Baxter equation is explored. In
Theorem~\ref{thm:YB} we generalize a result of Rump and produce a canonical
solution for each skew left brace.  Some reconstruction theorems
similar to those of Etingof, Schedler and Soloviev~\cite{MR1722951}, Lu, Yan and Zhu \cite{MR1769723} and Soloviev~\cite{MR1809284}
are given at the end of this section. 
The method for constructing classical
and skew braces is given in Section~\ref{construction}.  
Section~\ref{algorithms} discusses the algorithm 
that produces and enumerates classical and skew left
braces  
and some
consequences.  Problems, questions and conjectures are discussed in
Section~\ref{problems}.

\section{Skew left braces}
\label{braces}

Braces were introduced by Rump in~\cite{MR2278047} to study set-theoretical involutive
solutions of the Yang--Baxter equation. The following definition generalizes
braces to the non-commutative setting.

\begin{defn}
	A \emph{skew left brace} is a group $A$ (written multiplicatively) with an additional group
    structure given by $(a,b)\mapsto a\circ b$ such that 
    \begin{equation}
        \label{eq:non-commutative_brace}
        a\circ(bc)=(a\circ b)a^{-1}(a\circ c)
    \end{equation}
    holds for all $a,b,c\in A$, where $a^{-1}$ denotes the inverse of $a$ with
    respect to the group structure given by $(a,b)\mapsto ab$. 
\end{defn}

Of course Rump's left braces are examples of skew braces. These are
braces where the group $(A,\cdot)$ is abelian. 

\begin{defn}
    A \emph{homomorphism} between two skew left braces $A$ and $B$ is
    a map $f\colon A\to B$ such that $f(ab)=f(a)f(b)$ and $f(a\circ
    b)=f(a)\circ f(b)$ for all $a,b\in A$. The \emph{kernel} of $f$ is
    \[
        \ker f=\{a\in A:f(a)=1\},
    \]
	where $1$ denotes the identity of the group $(A,\cdot)$ with multiplication
	$a\cdot b=ab$ for all $a,b\in A$. 
\end{defn}

\begin{example}
	\label{exa:trivial}
	Let $(A,\cdot)$ be a group. Then $A$ is a skew left brace with
	$a\circ b=ab$ for all $a,b\in A$.  Similarly, $a\star b=ba$ defines a
	skew left brace structure over $A$. These braces are isomorphic
	if and only if $(A,\cdot)$ is abelian.
\end{example}

\begin{example}
	\label{exa:ds}
	Let $A$ and $B$ be groups and let $\alpha\colon A\to\Aut(B)$ be a
	group homomorphism. Then $A\times B$ has a skew left brace structure
	given by
    \begin{align*}
        &(a,b)(a',b')=(aa',bb'),\\
        &(a,b)\circ(a',b')=(aa',b\alpha_a(b')),
    \end{align*}
    where $a,a'\in A$ and $b,b'\in B$. 
\end{example}

\begin{example}
	\label{exa:sd}
    Let $A$ and $B$ be groups and let $\alpha\colon A\to\Aut(B)$ be a
    group homomorphism. Assume that $A$ is abelian. Then $A\times B$ has a
    skew left brace structure given by
    \begin{align*}
        &(a,b)(a',b')=(aa',b\alpha_a(b')),\\
        &(a,b)\circ(a',b')=(aa',bb'),
    \end{align*}
    where $a,a'\in A$ and $b,b'\in B$. 
\end{example}

\begin{example}
    \label{exa:WX}
    This example is motivated by the paper of Weinstein and
	Xu on the Yang--Baxter equation, see~\cite{MR1178147}. Let $A$ be a group
	and $A_+,A_-$ be subgroups of $A$ such that $A$ admits a \emph{unique
	factorization} as $A=A_+A_-$. Thus each $a\in A$ can be written in a unique
	way as $a=a_+a_-$ for some $a_+\in A_+$ and $a_-\in A_-$.  The map
	\[
		A_+\times A_-\to A,\quad
		(a_+,a_-{)}\mapsto a_+(a_{-})^{-1},
	\]
	is bijective. Using this map we transport the group structure of the direct
	product $A_+\times A_-$ into the set $A$. For $a=a_+a_-\in A$ and
	$b=b_+b_-\in A$ let 
	\begin{align*}
		a\circ b&=a_+ba_-.
	\end{align*}
	Then $(A,\circ)$ is a group. Furthermore, $A$ is a
	skew left brace.
\end{example}

\begin{lem}
    \label{lem:basic}
	Let $A$ be a skew left brace. Then the following properties hold:
    \begin{enumerate}
        \item $1=1_\circ$, where $1_\circ$ denotes the unit of the group $(A,\circ)$. 
        \item $a\circ(b^{-1}c)=a(a\circ b)^{-1}(a\circ c)$ for all $a,b,c\in A$.
        \item $a\circ(bc^{-1})=(a\circ b)(a\circ c)^{-1}a$ for all $a,b,c\in A$.
    \end{enumerate}
    
    \begin{proof}
		The first claim follows from~\eqref{eq:non-commutative_brace} with
		$c=1_{\circ}$.  To prove the second claim let $d=bc$.
		Then~\eqref{eq:non-commutative_brace} becomes $a\circ d =(a\circ
		b)a^{-1}(a\circ b^{-1}d)$ and the claim follows. The third claim is
		proved similarly.
    \end{proof}
\end{lem}

\begin{rem}
    \label{rem:formulas}
	Let $A$ be a skew left brace. For each $a\in A$ the map
	\[
	\lambda_a\colon A\to A,\quad
	b\mapsto a^{-1}(a\circ b),
	\]
	is bijective with inverse $\lambda^{-1}_a\colon A\to A$, $b\mapsto
		\overline{a}\circ (ab)$, where $\overline{a}$ is the inverse of $a$ with
	respect to $\circ$. It follows that 
    \[
        a\circ b=a\lambda_a(b),\quad
        ab=a\circ\lambda_a^{-1}(b)
    \]
    hold for all $a,b\in A$.
\end{rem}

The following proposition extends results of Rump~\cite{MR2278047} and
Gateva-Ivanova into the non-commutative setting, see~\cite[Proposition
3.3]{GI}. 

\begin{pro}
	\label{pro:GI}
	Let $A$ be a set and assume that $A$ has two operations
	such that $(A,\cdot)$ and $(A,\circ)$ are
	groups. Assume that $\lambda\colon A\to\Sym_A$, $a\mapsto\lambda_a$, is
	given by $\lambda_a(b)=a^{-1}(a\circ b)$. The following are equivalent:
	\begin{enumerate}
		\item $A$ is a skew left brace.
		\item $\lambda_{a\circ b}(c)=\lambda_a\lambda_b(c)$ for all $a,b,c\in A$.
		\item $\lambda_a(bc)=\lambda_a(b)\lambda_a(c)$ for all $a,b,c\in A$.
	\end{enumerate}

	\begin{proof}
		Let us first prove that $(1)\implies(2)$. Let $a,b,c\in A$. Since $A$
		is a brace and $a\circ b^{-1}=a(a\circ b)^{-1}a$ by Lemma~\ref{lem:basic}, 
		\begin{align*}
			\lambda_a\lambda_b(c)
			&=a^{-1}(a\circ\lambda_b(c))=a^{-1}(a\circ (b^{-1}(b\circ c)))\\
			&=a^{-1}(a\circ b^{-1})a^{-1}(a\circ b\circ c)
			=(a\circ b)^{-1}(a\circ b\circ c)
			=\lambda_{a\circ b}(c).
		\end{align*}

		Now we prove $(2)\implies(3)$. Since 
		$ab=a\circ\lambda^{-1}_a(b)$ for all $a,b\in A$, 
		\begin{align*}
			\lambda_a(bc)&=\lambda_a(b\circ\lambda^{-1}_b(c))
			=a^{-1}(a\circ b\circ\lambda^{-1}_b(c))\\
			&=a^{-1}(a\circ b)(a\circ b)^{-1}(a\circ b\circ\lambda^{-1}_b(c))\\
			&=\lambda_a(b)\lambda_{a\circ b}\lambda^{-1}_b(c)
			=\lambda_a(b)\lambda_a\lambda_b\lambda^{-1}_b(c)
			=\lambda_a(b)\lambda_a(c).
		\end{align*}
		
		Finally we prove that $(3)\implies(1)$. Let $a,b,c\in A$. Then
		\[
			a^{-1}(a\circ(bc))=\lambda_a(bc)=\lambda_a(b)\lambda_a(c)=a^{-1}(a\circ b)a^{-1}(a\circ c), 
		\]
		and hence $a\circ(bc)=(a\circ b)a^{-1}(a\circ c)$. 
	\end{proof}
\end{pro}

\begin{cor}
    \label{cor:lambda}
    Let $A$ be a skew left brace and 
    \[
    \lambda\colon
    (A,\circ)\to\Aut(A,\cdot),\quad 
    a\mapsto \lambda_a(b)=a^{-1}(a\circ b). 
    \]
    Then $\lambda$ is a group homomorphism. 

    \begin{proof}
		It follows immediately from Proposition~\ref{pro:GI}. 
    \end{proof}
\end{cor}

Let $A$ and $G$ be groups and assume that 
$G\times A\to A$, $(g,a)\mapsto g\cdot a$,
is a left action of $G$ on $A$ by automorphisms.  
A \emph{bijective
$1$-cocyle} is a bijective map $\pi\colon G\to A$ such that 
\begin{equation}
    \label{eq:1cocycle}
    \pi(gh)=\pi(g)(g\cdot \pi(h))
\end{equation}
for all $g,h\in G$. 

\begin{pro}
	\label{pro:1cocycle}
    Over any group $(A,\cdot)$ the following data are equivalent:
    \begin{enumerate}
        \item A group $G$ and a bijective
            $1$-cocycle $\pi\colon G\to A$. 
        \item A skew left brace structure over $A$. 
    \end{enumerate}

    \begin{proof}
        Consider on $A$ a second group structure given by 
        \[
		a\circ b=\pi(\pi^{-1}(a)\pi^{-1}(b))
		\]
		for all
        $a,b\in A$.  Since $\pi$ is a $1$-cocycle and $G$ acts on $A$ by
        automorphisms, 
        \begin{align*}
            a\circ (bc)&=\pi(\pi^{-1}(a)\pi^{-1}(bc))=a(\pi^{-1}(a)\cdot (bc))\\
            &=a( (\pi^{-1}(a)\cdot b)(\pi^{-1}(a)\cdot c))
            =(a\circ b)a^{-1}(a\circ c)
        \end{align*}
        holds for all $a,b,c\in A$.
        
        Conversely, assume that $A$ is a skew left brace. Set $G=A$ with
        the multiplication $(a,b)\mapsto a\circ b$ and $\pi=\id$. By
        Corollary~\ref{cor:lambda}, $a\mapsto\lambda_a$, is a group homomorphism and 
        hence $G$ acts on $A$ by automorphisms. Then~\eqref{eq:1cocycle} holds
        and therefore $\pi\colon G\to A$ is a bijective $1$-cocycle. 
    \end{proof}
\end{pro}

\begin{rem}
	The construction of Proposition~\ref{pro:1cocycle} is categorical. 
\end{rem}

\section{Ideals and quotients}
\label{ideals}

\begin{defn}
    Let $A$ be a skew left brace. A normal subgroup $I$ of $(A,\circ)$ is
    said to be an \emph{ideal} of $A$ if $Ia=aI$ and $\lambda_a(I)\subseteq I$
    for all $a\in A$. 
\end{defn}

\begin{example}
    Let $f\colon A\to B$ be a skew brace homomorphism. Then $\ker f$ is an ideal
    of $A$ since 
    \[
        f(\lambda_a(x))=\lambda_{f(a)}(f(x))=1
    \]
    for all $x\in\ker f$ and $a\in A$.
\end{example}

\begin{lem}
	Let $A$ be a skew left brace and $I\subseteq A$ be an ideal. Then the
	following properties hold:
    \begin{enumerate}
        \item $I$ is a normal subgroup of $(A,\cdot)$. 
        \item $a\circ I=aI$ for all $a\in A$. 
        \item $I$ and $A/I$ are skew braces.
    \end{enumerate}

    \begin{proof}
        Let $a,b\in I$. Then
        $a^{-1}b=\lambda_a(\overline{a}\circ b)\in I$ and hence  
        $I$ is a subgroup of $(A,\cdot)$. 
        Remark~\ref{rem:formulas} implies 
        \[
            aI=a\circ I=I\circ a=Ia
        \]
        for all $a\in A$. Thus $I$ is a normal subgroup of $(A,\cdot)$
        \and hence it follows that $I$ is a skew left brace.
        Since the quotient groups $A/I$ for both operations are the
        same, $A/I$ is a skew left brace. 
    \end{proof}
\end{lem}

\begin{defn}
	Let $A$ be a skew left brace. The \emph{socle} of $A$ is 
	\[
		\Soc(A)=\{a\in A:a\circ b=ab,\;b(b\circ a)=(b\circ a)b\text{ for all $b\in A$}\}.
	\]
\end{defn}

\begin{lem}
	\label{lem:socle}
	Let $A$ be a skew left brace. Then $\Soc(A)$ is an ideal of $A$
	contained in the center of $(A,\cdot)$. 

	\begin{proof}
		Let us first prove that $\Soc(A)$ is a subgroup of $(A,\circ)$.  Clearly
		$1\in\Soc(A)$. Let $a,a'\in A$ and $b\in A$. Then $a\circ a'\in\Soc(A)$
		since 
		\[
		(a\circ a')\circ b=a\circ(a'\circ b)=a\circ (a'b)=a(a'b)=(aa')b=(a\circ a')b.
		\]
		Now since $\overline{a}=a^{-1}\in\Soc(A)$ and $b=(aa^{-1})\circ b=a\circ
		(a^{-1}\circ b)=a(a^{-1}\circ b)$, it follows that
		$\overline{a}b=a^{-1}b=a^{-1}\circ b=\overline{a}\circ b$. Hence $\Soc(A)$
		is a subgroup of $(A,\circ)$. 

		A direct calculation proves that 
		\begin{equation}
			\label{eq:util}
			\lambda_b(a)=b\circ a\circ\overline{b}\quad\text{for all $a\in\Soc(A)$ and $b\in A$.}\\
		\end{equation}
		Then it follows that $\Soc(A)\subseteq\{a\in A:a\circ b=ab,\lambda_b(a)\circ b=b\circ a\text{ for all $b\in A$}\}$. 

		Let $a\in\Soc(A)$ and $b,c\in A$. Then 
		\begin{align*}
			&\lambda_c\lambda_b(a)=\lambda_{c\circ b}(a)=(c\circ b)\circ a\circ\overline{c\circ c}=c\circ\lambda_b(a)\circ\overline{c},\\
			&\lambda_b(a)c=b^{-1}(b\circ a)c=(b\circ a)b^{-1}c=b\circ(a(\overline{b}\circ c))=b\circ a\circ\overline{b}\circ c=\lambda_b(a)\circ c.
		\end{align*}
		Hence $\lambda_b(\Soc(A))\subseteq\Soc(A)$ for all $b\in A$ and $\Soc(A)$ is
		a normal subgroup of $(A,\circ)$ by~\eqref{eq:util}.

		Now we prove that $\Soc(A)$ is central in $(A,\cdot)$.
		Let $a\in \Soc(A)$, $b\in A$ and $c=\overline{b}$.  Since 
		\[
		c\circ (ba)=(c\circ b)c^{-1}(c\circ a)=c^{-1}(c\circ a)=(c\circ a)c^{-1}=c\circ (ab),
		\]
		it follows that $ba=ab$. 
	\end{proof}
\end{lem}

\section{Braces and the Yang--Baxter equation}
\label{YB}

We turn our attention to the connection between skew left braces and
set-theoretic solutions of the Yang--Baxter equation. The following theorem
generalizes a result of Rump to the non-commutative setting, see~\cite[Lemma
2]{MR3177933}. 

\begin{thm}
    \label{thm:YB}
    Let $A$ be a skew left brace. Then
    \begin{equation}
        \label{eq:braiding_operator}
		r_A\colon A\times A\to A\times A,
        \quad
		r_A(a,b)=(\lambda_a(b),\lambda^{-1}_{\lambda_a(b)}( (a\circ b)^{-1}a(a\circ b)),
    \end{equation}
	is a non-degenerate solution of the Yang--Baxter equation. Furthermore,
	$r_A$ is involutive if and only if $ab=ba$ for all $a,b\in A$.
\end{thm}

\begin{rem}
	\label{rem:braiding_operator}
	Recall from~\cite{MR1769723} that a \emph{braiding operator} over a group
	$(A,\circ)$ with multiplication $m\colon (a,b)\mapsto a\circ b$  is a bijective
	map $r\colon A\times A\to A\times A$ such that
	\begin{enumerate}
		\item $r(a\circ b,c)=(\id\times m)r_{12}r_{23}(a,b,c)$ for all $a,b,c\in A$,
		\item $r(a,b\circ c)=(m\times\id)r_{23}r_{12}(a,b,c)$ for all $a,b,c\in A$,
		\item $r(a,1)=(1,a)$ and $r(1,a)=(a,1)$ for all $a\in A$, and 
		\item $mr(a,b)=a\circ b$ for all $a,b\in A$.
	\end{enumerate}

	Braiding operators are equivalent to bijective $1$-cocycles by~\cite[Theorem 
    2]{MR1769723}, and bijective $1$-cocycles are equivalent to skew left
	braces by Proposition~\ref{pro:1cocycle}. One can prove
	that~\eqref{eq:braiding_operator} is the braiding operator corresponding to
	the skew left brace $A$ under this equivalence.
\end{rem}

\begin{proof}[Proof of Theorem~\ref{thm:YB}]
	Every braiding operator is a non-degenerate solution of the
	Yang--Baxter equation by~\cite[Corollary~3]{MR1769723}. Thus 
	it is enough to prove that $r_A$ is a braiding operator on
	$(A,\circ)$. Set $r=r_A$. Since $\lambda^{-1}_a(b)=\overline{a}\circ(ab)$ for all
	$a,b\in A$, 
	\begin{align*}
		\lambda^{-1}_{\lambda_a(b)}( (a\circ b)^{-1}a(a\circ b)) &=
		\overline{\lambda_a(b)}\circ (\lambda_a(b)(a\circ b)^{-1}a(a\circ b))
		=\overline{\lambda_a(b)}\circ( a \circ b)
	\end{align*}
	holds for all $a,b\in A$.  Thus $mr(a,b)=a\circ b$ for all $a,b\in A$.
	Clearly $r(a,1)=(1,a)$ and $r(1,a)=(a,1)$ for all $a\in A$. Let $a,b,c\in
	A$.  By Corollary~\ref{cor:lambda} one obtains 
	\begin{align*}
		(\id\times m)r_{12}r_{23}(a,b,c) &= (\id\times m)r_{12}(a,\lambda_b(c),\overline{\lambda_b(c)}\circ b\circ c)\\
		&=(\id\times m)(\lambda_a\lambda_b(c),\overline{\lambda_a\lambda_b(c)}\circ a\circ\lambda_b(c),\overline{\lambda_b(c)}\circ b\circ c)\\
		&=(\lambda_a\lambda_b(c),\overline{\lambda_a\lambda_b(c)}\circ a\circ b\circ c)
		=r(a\circ b,c).
	\end{align*}

	From Remark~\ref{rem:formulas} and Proposition~\ref{pro:GI} one obtains that 
	\[
		\lambda_a(b\circ c)=\lambda_a(b)\lambda_{a\circ b}(c)
	\]
	holds for all $a,b,c\in A$. From this formula one deduces that 
	\[
		\lambda_a(b)\circ\lambda_{\overline{\lambda_a(b)}\circ a\circ b}(c)
		=\lambda_a(b)\circ\lambda^{-1}_{\lambda_a(b)}\lambda_a\lambda_b(c)
		=\lambda_a(b)\lambda_{a\circ b}(c)
		=\lambda_a(b\circ c).
	\]
    holds for all $a,b,c\in A$. Then 
    \begin{align*}
        (m\times\id)r_{23}r_{12}(a,b,c) &= (m\times\id)r_{23}(\lambda_a(b),\overline{\lambda_a(b)}\circ a\circ b,c)\\
        &=(m\times\id)(\lambda_a(b),\lambda_{\overline{\lambda_a(b)}\circ a\circ b}(c),\overline{\lambda_{\overline{\lambda_a(b)}\circ a\circ b}}\circ \overline{\lambda_a(b)}\circ a\circ b\circ c)\\
        &=(\lambda_a(b)\circ\lambda_{\overline{\lambda_a(b)}\circ a\circ b}(c),\overline{\lambda_{\overline{\lambda_a(b)}\circ a\circ b}}\circ \overline{\lambda_a(b)}\circ a\circ b\circ c)\\
        &=(\lambda_a(b\circ c),\overline{\lambda_a(b\circ c)}\circ a\circ b\circ c)
        =r(a,b\circ c).
    \end{align*}
    for all $a,b,c\in A$.
\end{proof}

\begin{cor}
	\label{cor:YB}
	Let $A$ be a skew left brace and $X\subseteq A$ be a subset of
	$A$. Assume $b\lambda_a(x)b^{-1}\in X$ for all $x\in X$ and $a,b\in A$.
	Then $r_A|_{X\times X}$ is a non-degenerate solution of the Yang--Baxter equation.

	\begin{proof}
		Clearly $\lambda_a(x)\in X$ and $bxb^{-1}\in X$ for all $a,b\in A$ and
		$x\in X$. Then it follows that 
		$\lambda^{-1}_{\lambda_x(y)}((x\circ y)^{-1}x(x\circ y))\in X$ 
        for all $x,y\in X$. Now Theorem~\ref{thm:YB} implies the claim. 
	\end{proof}
\end{cor}

\begin{example}
	Let $A$ and $B$ be groups and $\alpha\colon A\to\Aut(B)$ be a group
	homomorphism. The skew left brace of Example~\ref{exa:ds} yields the
	following solution:
	\begin{align*}
		&r\colon (A\times B)\times(A\times B)\to(A\times B)\times(A\times B),\\
		&r((a_1,b_1)(a_2,b_2))=((a_2,\alpha_{a_1}(b_2)), (a_2^{-1}a_1a_2,b_2^{-1}\alpha_{a_2^{-1}}(b_1\alpha_{a_1}(b_2)))).
	\end{align*}
\end{example}

\begin{example}
	Let $A$ be an abelian group, $B$ be a group and $\alpha\colon A\to\Aut(B)$ be
	a group homomorphism. The skew left brace of Example~\ref{exa:sd}
	yields the following solution:
	\begin{align*}
		&r\colon (A\times B)\times(A\times B)\to(A\times B)\times(A\times B),\\
		&r((a_1,b_1)(a_2,b_2))=( (a_2,\alpha_{a_1^{-1}}(b_2)),(a_1,\alpha_{a_1^{-1}}(b_2^{-1})b_1b_2)).
	\end{align*}
\end{example}

\begin{example}
	Let $A$ be the skew left brace constructed in
	Example~\ref{exa:WX}.  The solution of 
	Theorem~\ref{thm:YB} 
	is similar to the solution constructed by Weinstein and Xu in terms of
	factorizable Poisson groups~\cite[Theorem~9.2]{MR1178147}. The latter is 
	%the solution given by Weinstein and Xu is 
	$\tau r_A\tau$, where $r_A$  
	is the solution of 
	Theorem~\ref{thm:YB} and $\tau(x,y)=(y,x)$ for all $x,y$. 
\end{example}

Based on~\cite{MR1769723}, for each skew left brace $A$ we relate the
solution $r$ given by Theorem~\ref{thm:YB} to the so-called Venkov
solution, i.e. 
\[
s(a,b)=(b,b^{-1}ab),\quad
a,b\in A.
\]

\begin{pro}
	\label{pro:guitar}
	Let $A$ be a skew left brace. For each $n\in\N$ the map $T_n$
	given by
	\[
		T_n(a_1,\dots,a_{n-1},a_n)=(a_1,\lambda_{a_1}(a_2),\lambda_{a_1\circ a_2}(a_3),\dots,\lambda_{a_1\circ\cdots\circ a_{n-1}}(a_n)) 
	\]
	is invertible and satisfies 
	\begin{equation}
		\label{eq:equivalence}
		T_nr_{i,i+1}=s_{i,i+1}T_n
	\end{equation}
	for all $n\geq2$ and $i\in\{1,\dots,n-1\}$, where $r_{i,i+1}$ and
	$s_{i,i+1}$ denote the actions of the braid group $\B_n$ on $A^n=A\times\cdots\times A$
	($n$-times) induced from $r$ and $s$ respectively. 

	\begin{proof}
		A direct calculation shows that $T_n$ is invertible with inverse 
		\[
			T_n^{-1}(a_1,\dots,a_n)=(a_1,\lambda_{a_1}^{-1}(a_2),\lambda^{-1}_{a_1a_2}(a_3),\dots,\lambda^{-1}_{a_1\dots a_{n-1}}(a_n)).
		\]

		To
		prove~\eqref{eq:equivalence} we proceed by induction on $n$. The case
        $n=2$ follows from a direct calculation since
        \begin{align*}
            T_2r_{12}(a,b)&=T_2(\lambda_a(b),\lambda^{-1}_{\lambda_a(b)}( (a\circ b)^{-1}a(a\circ b)))=(\lambda_a(b),(a\circ b)^{-1}a(a\circ b)),\\
            &=(\lambda_a(b),\lambda_a(b)^{-1}a\lambda_a(b))=s_{12}(a,\lambda_a(b))=s_{12}T_2(a,b)
        \end{align*}
        holds for all $a,b\in A$.
        So assume that the claim holds
		for $n-1$. Since $T_nr_{1,2}=s_{1,2}T_n$ is the same as $T_2r=sT_2$, we
		need to prove~\eqref{eq:equivalence} for all $i\in\{2,\dots,n-1\}$.
		Write
		\[
			T_n=U_n(\id\times T_{n-1}),
		\]
		where 
		\[
			U_n(a_1,\dots,a_{n-1},a_n)=(a_1,\lambda_{a_1}(a_2),\dots,\lambda_{a_1}(a_{n-1}),\lambda_{a_1}(a_n)).
		\]
		Since each $\lambda_a$ is an automorphism of $(A,\cdot)$, it follows
		that $U_n s_{i,i+1}=s_{i,i+1}U_n$ for $i\geq2$ and
		hence~\eqref{eq:equivalence} holds for all $i\geq2$. 
	\end{proof}
\end{pro}

\begin{rem}
	Proposition~\ref{pro:guitar} also follows from~\cite[Proposition~6.2]{LV}.  The map
	$T_n$ is the so-called \emph{guitar map}, see for example~\cite[\S6]{LV}.
\end{rem}

The universal construction of Lu, Yan and Zhu, given
in~\cite[Theorem~9]{MR1769723} can be restated in the language of skew
left braces.  
This was done by Rump in the case of involutive
solutions, see~\cite{MR2278047}.
Recall that the \emph{enveloping (or structure) group} of
a solution $(X,r)$ is the group $G(X,r)$ generated by the elements of $X$
with relations
\[
	x\circ y=\sigma_x(y)\circ\tau_y(x),\quad x,y\in X.
\]
Let $\iota\colon X\to G(X,r)$ be the canonical map. 

\begin{thm}
	\label{thm:LYZ}
	Let $X$ be a set, $r\colon X\times X\to X\times X$,
	$r(x,y)=(\sigma_x(y),\tau_y(x))$ be a non-degenerate solution of the
	Yang--Baxter equation. 
    Then there exists a unique skew left brace structure over $G(X,r)$ such that its
	associated solution $r_{G}$ satisfies 
	\[
	r_{G}(\iota\times\iota)=(\iota\times\iota)r.
	\]	
    Furthermore, if $B$ is a skew left brace and $f\colon X\to B$ is a map such that
	$(f\times f)r=r_B(f\times f)$, then there exists a unique group homomorphism
	$\phi\colon G(X,r)\to B$ such that $f=\phi\iota$ and
	$(\phi\times\phi)r_{G}=r_B(\phi\times\phi)$. 
\end{thm}

\begin{proof}
	The claim follows from the universal construction of~\cite[Theorem~9]{MR1769723} and the equivalence between braiding operators and
	skew braces, see Remark~\ref{rem:braiding_operator}.
\end{proof}

The following corollary is essentially~\cite[Theorem~2.6]{MR1809284}. 

\begin{cor}
	Let $(X,r)$ be a finite non-degenerate solution of the Yang--Baxter
    equation. Then $G(X,r)/\Soc(G(X,r))$ is a finite skew left brace. 

	\begin{proof}
		It follows from Theorem~\ref{thm:LYZ} and Lemma~\ref{lem:socle}.
	\end{proof}
\end{cor}

\section{Constructing skew braces}
\label{construction}

Let $A$ be a group.  The \emph{holomorph} of $A$ is the group
$\Hol(A)=\Aut(A)\ltimes A$, where the product is given by
\[
(f,a)(g,b)=(fg,af(b))
\]
for all $a,b\in A$ and $f,g\in\Aut(A)$. Any subgroup $H$ of $\Hol(A)$ acts on
$A$ 
\begin{equation} 
	\label{eq:right_action}
	%a\triangleleft (f,x)=\pi_2( (\id,a)(f,x) )=f(a)x,\quad a,x\in A,\,f\in \Aut(A),
	(f,x)\cdot a=\pi_2( (f,x)(\id,a) )=xf(a),\quad a,x\in A,\,f\in \Aut(A),
\end{equation}
where $\pi_2\colon \Hol(A)\to A$, $(f,a)\mapsto a$.  In particular $\Hol(A)$
acts transitively on $A$ and the stabilizer of any $a\in A$ is isomorphic to
$\Aut(A)$.  

Recall that a subgroup $H$ of $\Hol(A)$ is \emph{regular} if for each $a\in A$
there exists a unique $(f,x)\in H$ such that $xf(a)=1$. The following result is well-known.

\begin{lem}
	\label{lem:pi_2|_H}
    Let $A$ be a group and $H$ be a regular subgroup of $\Hol(A)$. Then
    $\pi_2|_H\colon H\to A$, $(f,a)\mapsto a$, is bijective. 

    \begin{proof}
		We first prove that $\pi_2|_H$ is injective. Let $(f,a),(g,b)\in H$ be
		such that $\pi_2(f,a)=\pi_2(g,b)$. Then $a=b$.  Since $H$ is a
		subgroup, 
		\[
		(f,a)^{-1}=(f^{-1},f^{-1}(a^{-1}))\in H,
		\quad
		(g,a)^{-1}=(g^{-1},g^{-1}(a^{-1}))\in H, 
		\]
		and hence $f=g$ since
		$f^{-1}(a)f^{-1}(a^{-1})=g^{-1}(a)g^{-1}(a^{-1})=1$ and $H$ is a
		regular subgroup. 
		
		Now we prove that $\pi_2|_H$ is surjective. Let $a\in A$. The
		regularlity of $H$ implies the existence of an automorphism
		$f\in\Aut(A)$ such that $(f,f(a^{-1}))\in H$.  Then $(f^{-1},a)\in H$
		and the claim follows.
    \end{proof}
\end{lem}

The following theorem goes back to Bachiller~\cite{MR3465351}.  The proof in
our case is the same as for braces.  It is a generalization of a result of
Catino and Rizzo~\cite{MR2486886}.

\begin{thm}
	\label{thm:Bachiller}
	Let $A$ be skew left brace. Then $\{(\lambda_a,a):a\in A\}$ is a
	regular subgroup of $\Hol(A,\cdot)$. Conversely, if $(A,\cdot)$ is a group and $H$ is a
	regular subgroup of $\Hol(A,\cdot)$, then $A$ is a skew left brace with 
	$(A,\circ)\simeq H$, where 
    \[
    a\circ b=af(b)
    \]
	and $(\pi_2|_H)^{-1}(a)=(f,a)\in H$. 

    \begin{proof}
        Since $\lambda$ is a group homomorphism and
$a\lambda_a(b)=a\circ b$ for all $a,b\in A$, it follows that
$\{(\lambda_a,a):a\in A\}$ is a subgroup of $\Hol(A,\cdot)$.  Since
$(A,\circ)$ is a group, the regularlity also follows.

		Assume now that $H$ is a regular subgroup. By Lemma~\ref{lem:pi_2|_H},
		$\pi_2|_H$ is bijective.  Use the bijection $\pi_2|_H$ to transport the
		product of $H$ into $A$: 
		\[
        a\circ b=\pi_2|_H\left((\pi_2|_H)^{-1}(a)(\pi_2|_H)^{-1}(b)\right)=af(b),
		\]
		where $a,b\in A$ and $(\pi_2|_H)^{-1}(a)=(f,a)\in H$. 
		Then $(A,\circ)$ is a group and $A$ is a skew left brace since 
		\[
		a\circ (bc)=af(bc)=af(b)f(c)=af(b)a^{-1}af(c)=(a\circ b)a^{-1}(a\circ c)
		\]
		holds for all $a,b,c\in A$. 
    \end{proof}
\end{thm}

\begin{pro}
	Let $A$ be a group.  There exists a bijective correspondence between
	skew left brace structures over $A$ and regular subgroups of
	$\Hol(A)$. Moreover, isomorphic skew braces structures over $A$
	correspond to conjugate subgroups of $\Hol(A)$ by elements of $\Aut(A)$.

	\begin{proof}
        Assume that the group $A$ has two skew left brace structures given by
		$(a,b)\mapsto a\circ b$ and $(a,b)\mapsto a\times b$ and that
		$\phi\in\Aut(A, \cdot)$ satisfies $\phi(a\circ b)=\phi(a)\times\phi(b)$ for
		all $a,b\in A$.  We claim that $\{(\lambda_a,a):a\in A$\} and
		$\{(\mu_a,a):a\in A\}$, where $\lambda_a(b)=a^{-1}(a\circ b)$ and
		$\mu_a(b)=a^{-1}(a\times b)$, are conjugate by $\phi$. Since
        \[
        \phi\lambda_a\phi^{-1}(b)=\phi(a^{-1}(a\circ\phi^{-1}(b))=\phi(a)^{-1}(\phi(a)\times b)=\mu_{\phi(a)}(b),
        \]
        one obtains that $\phi(\lambda_a,a)\phi^{-1}=(\mu_{\phi(a)},\phi(a))$
        and hence the claim follows.

		Conversely, let $H$ and $K$ be regular subgroups of $\Hol(A)$ and 
		asssume that there exists $\phi\in\Aut(A, \cdot)$ such that $\phi^{-1}
		H\phi=K$. Let
		$(f,a)=(\pi_2|_H)^{-1}(a)\in H$, $(g,a)=(\pi_2|_K)^{-1}(a)\in K$ 
		and write 
        $a\circ b=af(b)$
        and 
        $a\times b=ag(b)$. Since 
        $\phi(f,a)\phi^{-1}=(\phi f\phi^{-1},\phi(a))\in K$, it follows that 
        $(\pi_2|_K)^{-1}(\phi(a))=(\phi f\phi^{-1}, \phi(a))$. Then, since $\phi\in\Aut(A,\cdot)$,  
        \[
            \phi(a)\times\phi(b)=\phi(a) (\phi f\phi^{-1})(\phi(b))=\phi(a)\phi(f(b))=\phi(af(b))=\phi(a\circ b)
        \]
        and hence the skew left braces corresponding to $H$ and $K$ are
		isomorphic.
	\end{proof}
\end{pro}

\section{Computational results}
\label{algorithms}
\label{results}

We first present the algorithm used to enumerate skew left brace
structures over a given group $A$. The algorithm uses
Theorem~\ref{thm:Bachiller}.

\begin{algorithm}
	\label{alg:Bachiller}
	Let $A$ be a finite group.  To construct all skew left brace
	structures over $A$ we proceed as follows:
	\begin{enumerate}
		\item Compute the holomorph $\Hol(A)$ of $A$.
		\item Compute the list of regular subgroups of $\Hol(A)$ of
			order $|A|$ up to conjugation by elements of $\Aut(A)$.
		\item For each representative $H$ of regular subgroups of
			$\Hol(A)$ construct the map $p\colon A\to H$ given by $a\mapsto
            (f,f(a)^{-1})$, where $(f,f(a)^{-1})\in H$. The triple $(H,A,p\colon A\to H)$
			yields a skew left brace structure over $A$ with
			multiplication given by 
			$a \circ b=p^{-1}(p(a)p(b))$ for all $a,b\in A$.
	\end{enumerate}
\end{algorithm}

\begin{rem}
	To enumerate all skew left brace structures over $A$ 
	the third step of Algorithm~\ref{alg:Bachiller} is not needed.  
\end{rem}

\begin{rem}
	Recall that a left brace is an abelian group $(A,+)$ 
	with another group structure, defined via the multiplication
	$(a,b)\mapsto ab$ such that $a(b+c)+a=ab+ac$ holds for all $a,b,c\in A$.
	Left braces with additive group isomorphic to a given group $A$ can be
	constructed applying Algorithm~\ref{alg:Bachiller} to the abelian group
	$A$. In the second step of Algorithm~\ref{alg:Bachiller} it is enough to
	compute the list of regular solvable subgroups of $\Hol(A)$, since the
	multiplicative group of a left brace is solvable by~\cite[Proposition~2.5]{MR1722951}.
\end{rem} 

Algorithm~\ref{alg:Bachiller} was implemented both in $\GAP$ and $\MAGMA$ with
different performances and were run on a Intel(R) Core(TM) i5-4440 CPU @3.10GHz
with $16$gb of RAM, under Linux. 

\subsection{Skew left braces} 

For $n\in\N$ let $c(n)$ be the number of non-isomorphic skew left
braces of size $n$.  

The number of skew left braces of size
$n\leq30$ has been determined using Algorithm~\ref{alg:Bachiller}.  Table
\ref{tab:c(n)} shows some values of $c(n)$. The calculation took about
twenty minutes.

\begin{table}[h] 
    \caption{The number of non-isomorphic skew left braces.}
    \begin{tabular}{|c|ccccccccccccccc|}
		\hline
        $n$    & 1 & 2 & 3 & 4 & 5  & 6 & 7 & 8 & 9 & 10 & 11 & 12 & 13 & 14 & 15 \tabularnewline
        $c(n)$ & 1 & 1 & 1 & 4 & 1  & 6 & 1 & 47 & 4  & 6  & 1  & 38 & 1 & 6 & 1 \tabularnewline
		\hline
        $n$    & 16 & 17 & 18 & 19 & 20 & 21  & 22 & 23 & 24 & 25 & 26 & 27 & 28 & 29 & 30 \tabularnewline
        $c(n)$ & 1605 & 1 & 49 & 1 & 43 & 8 & 6 & 1 & 855 & 4 & 6 & 101 & 29 & 1 & 36 \tabularnewline
		\hline
    \end{tabular}
    \label{tab:c(n)}
\end{table}

\subsection{Left braces} 

For $n\in\N$ let $b(n)$ be the number of non-isomorphic left braces of size
$n$.  

The number of left braces (up to isomorphism) of size $n\leq120$ has been
determined using Algorithm~\ref{alg:Bachiller}.  Table \ref{tab:b(n)} shows
some values of $b(n)$ and Table~\ref{tab:time} gives runtimes for
our~\MAGMA~implementation for some examples. The construction of left braces
requires considerably more CPU time, see Table~\ref{tab:construction} for some
examples.

\begin{table}[h]
	\caption{Some runtimes for enumerating left braces of size $n$.}
	\begin{tabular}{|r|r|r|}
		\hline
		$n$ & CPU time & $b(n)$ \tabularnewline
		\hline
		16 & 1 hour & 357 \tabularnewline
        48 & 18 hours & 1708 \tabularnewline
		54 & 5 minutes & 80\tabularnewline
		72 & 1 hour & 489 \tabularnewline
		80 & 17 hours & 1985\tabularnewline
		100 & 15 secs & 51 \tabularnewline
        108 & 28 hours & 494\tabularnewline
		112 & 12 hours & 1671\tabularnewline
		\hline
	\end{tabular}
	\label{tab:time}
\end{table}

\begin{table}[h]
        \caption{The number of non-isomorphic left braces.}
        \begin{tabular}{|c|cccccccccccc|}
			\hline
			$n$ & 1 & 2 & 3 & 4 & 5  & 6 & 7 & 8 & 9 & 10 & 11 & 12\tabularnewline
			$b(n)$ & 1  & 1  & 1  & 4  & 1  & 2  & 1  & 27  & 4  & 2  & 1  & 10\tabularnewline
			\hline
			$n$ & 13 & 14 & 15 & 16 & 17 & 18 & 19 & 20 & 21 & 22 & 23 & 24\tabularnewline
			$b(n)$ & 1  & 2  & 1  & 357  & 1  & 8  & 1  & 11  & 2  & 2  & 1  & 96\tabularnewline
			\hline
			$n$ & 25 & 26 & 27 & 28 & 29 & 30 & 31 & 32 & 33 & 34 & 35 & 36\tabularnewline
			$b(n)$ & 4  & 2  & 37  & 9  & 1  & 4 & 1  & ?  & 1  & 2  & 1  & 46\tabularnewline
			\hline
			$n$ & 37 & 38 & 39 & 40 & 41 & 42 &  43 & 44 & 45 & 46 & 47 & 48\tabularnewline
			$b(n)$	& 1  & 2  & 2  & 106  & 1  & 6  & 1  & 9  & 4 & 2  & 1  & 1708\tabularnewline
			\hline
			$n$    & 49 & 50 & 51 & 52 & 53 & 54 & 55 & 56 & 57 & 58 & 59 & 60\tabularnewline
			$b(n)$ &  4 & 8 & 1 & 11 & 1 & 80 & 2 & 91 & 2 & 2 & 1 & 28\tabularnewline
			\hline
			$n$ & 61 & 62 & 63 & 64 & 65 & 66 & 67 & 68 & 69 & 70 & 71 & 72\tabularnewline
			$b(n)$ & 1 & 2 & 11 & ? & 1 & 4 & 1 & 11 & 1 & 4 & 1 & 489\tabularnewline
			\hline
			$n$ & 73 & 74 & 75 & 76 & 77 & 78 & 79 & 80 & 81 & 82 & 83 & 84\tabularnewline
			$b(n)$ & 1 & 2 & 5 & 9 & 1 & 6 & 1 & 1985 & ? & 2 & 1 & 34\tabularnewline
			\hline
			$n$ & 85 & 86 & 87 & 88 & 89 & 90 & 91 & 92 & 93 & 94 & 95 & 96\tabularnewline
			$b(n)$ & 1 & 2 & 1 & 90 & 1 & 16 & 1 & 9 & 2 & 2 & 1 & ?\tabularnewline
			\hline
			$n$ & 97 & 98 & 99 & 100 & 101 & 102 & 103 & 104 & 105 & 106 & 107 & 108\tabularnewline
			$b(n)$ & 1 & 8 & 4 & 51 & 1 & 4 & 1 & 106 & 2 & 2 & 1 & 494 \tabularnewline
			\hline
			$n$ & 109 & 110 & 111 & 112 & 113 & 114 & 115 & 116 & 117 & 118 & 119 & 120\tabularnewline
			$b(n)$ & 1 & 6 & 2 & 1671 & 1 & 6 & 1 & 11 & 11 & 2 & 1 & 395\tabularnewline
			\hline
        \end{tabular}
        \label{tab:b(n)}
\end{table}

With current computational resources, we were not able to compute the number of
non-isomorphic left braces of orders $32$, $64$, $81$ and $96$. 

\begin{table}[h]
    \caption{Some runtimes for constructing left braces of size $n$.}
	\begin{tabular}{|r|r|r|}
		\hline
		$n$ & CPU time & $b(n)$ \tabularnewline
		\hline
		16 & 3 hours & 357 \tabularnewline
		54 & 40 minutes & 80\tabularnewline
%		56 & 14 secs & 91\tabularnewline 
		72 & 24 hours & 489 \tabularnewline
%		88 & 22 secs & 90\tabularnewline
%		100 & 106 secs & 51 \tabularnewline
%		104 & 40 secs & 106 \tabularnewline
		112 & 5 days & 1671\tabularnewline
		\hline
	\end{tabular}
	\label{tab:construction}
\end{table}

\subsection{Two-sided left braces (radical rings)}

Recall that a brace $B$ is a \emph{two-sided brace} if $(a+b)c+c=ac+bc$ holds
for all $a,b,c\in B$.  Two-sided braces are in bijective correspondence with
radical rings~\cite{MR2320986}.  Recall that a non-zero \emph{radical ring} is
a ring $R$ without identity such that for each $x\in R$ there is $y\in R$ such
that $x+y+xy=0$. Assume that $R$ is a radical ring. Then the \emph{circle operation},
\[
a\circ b=ab+a+b,
\quad
a,b\in R,
\]
makes $(R,+,\circ)$ into a two-sided brace.
Conversely, if $A$ is a two-sided brace, the operation $a*b=ab-a-b$, $a,b\in A$
makes $(A,+,*)$, into a radical ring.

To test whether a left brace is a two-sided
brace one has the following lemma of Gateva-Ivanova, 
see~\cite[Corollary~3.5]{GI}.

\begin{lem}[Gateva-Ivanova]
	\label{lem:GI}
	Let $A$ be a left brace. Then $A$ is a two-sided brace if and only if 
	\[
		bc\lambda^{-1}_{abc}(c)=c\lambda^{-1}_{ac}(\lambda_a(b)c)
	\]
	for all $a,b,c\in A$.
\end{lem}

For $n\in\N$ let $t(n)$ be the number of non-isomorphic two-sided braces of
size $n$.  Using the database of left braces constructed with
Algorithm~\ref{alg:Bachiller} and Lemma~\ref{lem:GI} one computes 
$t(n)$.  Table \ref{tab:two_sided} shows the value of $t(n)$ for
$n\leq24$. 

\begin{table}[h] 
    \caption{The number of non-isomorphic two-sided braces.}
    \begin{tabular}{|c|cccccccccccc|}
		\hline
        $n$    & 1 & 2 & 3 & 4 & 5  & 6 & 7 & 8 & 9 & 10 & 11 & 12\tabularnewline
        $t(n)$ & 1 & 1 & 1 & 4 & 1  & 1  & 1  & 22  & 4  & 1  & 1  & 4\tabularnewline
        \hline
        $n$ & 13 & 14 & 15 & 16 & 17 & 18 & 19 & 20 & 21 & 22 & 23 & 24\tabularnewline
        $t(n)$ & 1  & 1  & 1  & 221  & 1  & 4  & 1  & 4  & 1  & 1  & 1  & 22\tabularnewline
        \hline
%        $n$ & 25 & 26 & 27 & 28 & 29 & 30 & 31 & 32 & 33 & 34 & 35 & 36\tabularnewline
%        $t(n)$ & 4  & 1  & 28  & 4  & 1  & 1 & 1  & ?  & 1  & 1  & 1  & 16\tabularnewline
%		\hline
    \end{tabular}
    \label{tab:two_sided}
\end{table}

\begin{rem}
    For information on square-free two-sided braces, see~\cite{MR3177933}. These braces are defined by nilpotent groups of class $\leq2$.
\end{rem}

\section{Further questions}
\label{problems}

In this section we collect some questions and conjectures that appear naturally
after inspecting Table~\ref{tab:b(n)}.

\subsection{Left braces}

We first collect some problems and conjectures related to the number of left
braces.  

\begin{problem}
	Compute $b(32)$, $b(64)$, $b(81)$ and $b(96)$. 
\end{problem}

Table~\ref{tab:b(n)} suggests the following conjectures.

\begin{conjecture}
    \label{con:b(4p)}
    Let $p>3$ be a prime number. Then 
    \[
        b(4p)=\begin{cases}
            11 & \text{if $m\equiv1\bmod4$,}\\
            9 & \text{if $m\equiv3\bmod4$.}
        \end{cases}
    \]
\end{conjecture}

\begin{conjecture}
    \label{con:b(9p)}
    Let $p>3$ be a prime. Then
    \[
    b(9p)=\begin{cases}
        14 & \text{if $p\equiv1\bmod9$,}\\
        4 & \text{if $p\equiv 2,5\bmod9$,}\\
        11 & \text{if $p\equiv 4,7\bmod9$.}
    \end{cases}
    \]
\end{conjecture}

\begin{conjecture}
    \label{con:b(p^2q)}
    Let $p,q$ be prime numbers such that $p<q$ and $q\not\equiv1\bmod p$. Then
    $b(p^2q)=4$. 
\end{conjecture}

We have used~\cite{MR1935567} and computer calculations to show that
Conjectures~\ref{con:b(4p)}~and~\ref{con:b(9p)} are true up to $p=997$.  
In~\cite{Smoktunowicz}, Agata Smoktunowicz proved that Conjecture 6.4 is true.

%Conjecture~\ref{con:b(p^2q)} was verified for several pairs of primes $p,q$ such that $p<q<100$. 

\subsection{Quaternionic braces}

We now consider an important family of braces. 
Recall that for $m\in\N$ the \emph{generalized quaternion group} is
the group
\[
	Q_{4m}=\langle a,b:a^m=b^2,\,a^{2m}=1,\,b^{-1}ab=a^{-1}\rangle.
\]

\begin{defn}
    A brace is a \emph{quaternion brace} if its multiplicative group is
    isomorphic to some quaternion group. 
\end{defn}

\begin{conjecture}
    \label{con:quaternion}
   For $m\in\N$ let $q(4m)$ be the number of isomorphism classes of
   quaternion braces of size $4m$. Then for $m>2$
   \[
        q(4m)=\begin{cases}
            2 & \text{if $m$ is odd,}\\
            7 & \text{if $m\equiv0\bmod8$,}\\
            9 & \text{if $m\equiv4\bmod8$,}\\
            6 & \text{if $m\equiv2\bmod8$ or $m\equiv6\bmod8$}.\\
        \end{cases}
   \]
\end{conjecture}

We have checked Conjecture~\ref{con:quaternion} for all $m\leq512$. It
seems natural to ask the following questions.

\begin{question}
	Which finite abelian groups appear as the additive group of a quaternion
	brace?
\end{question}

For $m\in\{2,\dots,512\}$ the additive group of a quaternion brace of size $m$
is isomorphic to one of the following groups: 
\begin{equation*}
	\Z_{4m},\,
	\Z_{2m}\times\Z_2,\,
	\Z_m\times\Z_2\times\Z_2,\,
	\Z_m\times\Z_4,\,
	\Z_{m/2}\times\Z_2\times\Z_2\times\Z_2.
\end{equation*}
By inspection, one sees that the groups $\Z_{m}\times\Z_2^2$ appear whenever
$m\equiv2,4,6\bmod8$ and the groups $\Z_m\times\Z_4$ and $\Z_{m/2}\times\Z_2^3$
appear whenever $m\equiv4\bmod8$.

\begin{question}
   For $m>2$ let $A$ be a finite abelian group size $4m$. Compute the number of
   isomorphism classes of quaternion braces of size $4m$ with additive group
   isomorphic to $A$. 
\end{question}

In~\cite[\S5]{BCJ} quaternion braces of size $2^k$ are mentioned as an
important class of braces which could be useful to classify a certain family of
involutive non-degenerate solutions of the Yang--Baxter equation.
Conjecture~\ref{con:quaternion} implies the following:

\begin{conjecture}
    \label{con:Q}
    There are seven classes of isomorphism of quaternion braces of size $2^k$
    for $k>4$. 
\end{conjecture}

Conjecture~\ref{con:Q} was verified for all $k\in\{5,6,7,8,9\}$.
Table~\ref{tab:2^k} sums up our findings related to this important subclass of
braces. 

\begin{table}[h]
\caption{Number of braces with multiplicative group isomorphic to the
quaternion group $Q_{2^k}$ with $k>4$.}
\begin{tabular}{|c|c|}
    \hline
	Additive Group & Number of Braces\tabularnewline
	\hline
	$\Z_{2^{k}}$ & $1$\tabularnewline
	\hline
	$\Z_{2^{k-1}}\times\Z_2$ & $6$\tabularnewline
    \hline
\end{tabular}
\label{tab:2^k}
\end{table}

\begin{rem}
	The classification of left braces over cyclic groups was done by Rump
    in~\cite{MR2298848}. He proved that if a left brace $A$ has additive group
    isomorphic to $\Z/p^k$, where $p>2$ is a prime number, 
    then $(A,\cdot)\simeq\Z/p^k$. According to~\cite{BCJ}, the converse holds
    for all $p$.
    In~\cite{MR3320237}, Bachiller classified
	left braces of size $p^2$ and $p^3$, where $p$ is a prime number. 
    The
	techniques used in these papers might prove useful to address the
	questions, problems and conjectures in this section.
\end{rem}

\section*{Acknowledgements}

We thank Leandro del Pezzo for the computer where some calculations were
performed, David Bachiller for his suggestions, and Victoria Lebed for several
useful comments and corrections. We also thank
the referees for their careful review of our manuscript and valuable suggestions.

%\bibliographystyle{abbrv}
%\bibliography{refs}
\def\cprime{$'$}

\end{document}